\documentclass[11pt]{amsart}
\usepackage{amsfonts,amssymb,amscd,amsmath}
\usepackage{graphicx}
\theoremstyle{plain}
\newtheorem{theorem}{Theorem}[section]
\newtheorem{lemma}[theorem]{Lemma}
\newtheorem{corollary}[theorem]{Corollary}
\newtheorem{proposition}[theorem]{Proposition}

\theoremstyle{definition}
\newtheorem{definition}[theorem]{Definition}
\newtheorem{remark}[theorem]{Remark}
\newtheorem{example}[theorem]{Example}

\numberwithin{equation}{section}

\begin{document}

\title[Weighted H\"{o}lder-Zygmund spaces and the wavelet transform]{New classes of weighted H\"{o}lder-Zygmund spaces and the  wavelet transform}
\author[S. Pilipovi\'{c}]{Stevan Pilipovi\'{c}}
\address{University of Novi Sad\\ Department
 of Mathematics and Informatics\\ Trg Dositeja Obradovi\' ca 4\\ 21000 Novi Sad \\ Serbia }
\email{pilipovics@yahoo.com}

\author[D. Raki\'{c}]{Du\v{s}an Raki\'{c}}
\address{University of Novi Sad\\ Faculty of Technology\\ Bul. Cara Lazara 1\\ 21000 Novi Sad \\ Serbia }
\email{drakic@tf.uns.ac.rs}

\author[J. Vindas]{Jasson Vindas}
\address{Department of Mathematics\\ Ghent University\\
Krijgslaan 281 Gebouw S22\\ B 9000 Gent\\ Belgium }
\email{jvindas@cage.ugent.be}


\subjclass[2000]{26A12, 26B35, 42C40, 46E35, 46F05}
\keywords{wavelet transform; Zygmund spaces; H\"{o}lder spaces;
distributions; quasiasymptotics; slowly varying functions; generalized Littlewood-Paley pairs}
\thanks{S. Pilipovi\'{c} acknowledges support by the Project 174024 of the Serbian Ministry of Education and Sciences}
\thanks{D. Raki\'{c} acknowledges support by the Project III44006 of the Serbian Ministry of Education and Sciences and by the Project 114-451-2167 of the Provincial Secretariat for Science and Technological Development}
\thanks{J. Vindas acknowledges support by a Postdoctoral Fellowship of the Research
Foundation--Flanders (FWO, Belgium)}

\begin{abstract}
We provide a new and elementary proof of the continuity theorem for the wavelet and
left-inverse wavelet transforms on the spaces $ \mathcal{S}_0
(\mathbb{R}^{n}) $ and $ \mathcal{S}(\mathbb {H}^{n+1})$. We then introduce and study a new class of weighted H\"{o}lder-Zygmund spaces, where the weights are regularly varying functions. The analysis of these spaces is carried out via the wavelet transform and generalized Littlewood-Paley pairs.
\end{abstract}
\maketitle

\section{Introduction}

The purpose of this article is two folded.  The main one is to
define and analyze a new class of weighted H\"{o}lder-Zygmund
spaces via the wavelet transform \cite{hol1,jaffard-meyer,meyer}.
It is well known \cite{hol94,hol1,horm,stein} that the wavelet
transforms of elements of the classical Zygmund space
$C^{\alpha}_{\ast}(\mathbb{R}^{n})$ satisfy the size estimate $
\left|\mathcal W_\psi f(x,y)\right|\leq Cy^{\alpha}$, which, plus a
side condition, essentially characterizes the space itself. We
will replace the regularity measurement $y^{\alpha}$ by weights
from the interesting class of regularly varying functions
\cite{bingham,seneta}. Familiar functions such as $y^{\alpha},
y^{\alpha}\left|\log y\right|^{\beta}, y^{\alpha}\left|\log
\left|\log y\right|\right|^{\beta}$, ..., are regularly varying.

The continuity of the wavelet transform and its left-inverse  on test function spaces \cite{hol93} plays a very important role when analyzing many function and distribution spaces \cite{hol1}, such as the ones introduced in this article. Our second aim is to provide a new proof of the continuity theorem, originally obtained in \cite{hol93}, for these transforms on the function spaces $\mathcal{S}_0 (\mathbb{R}^{n}) $ and $\mathcal{S}(\mathbb {H}^{n+1})$. Our approach to the proof is completely elementary and substantially simplifies the much longer original proof from \cite{hol93} (see also \cite[Chap. 1]{hol1}).

The definition of our weighted Zygmund spaces is based on the useful concept of (generalized) Littlewood-Paley pairs, introduced in Subsection \ref{LPpairs}, which generalizes the familiar notion of (continuous) Littlewood-Paley decomposition of the unity \cite{horm}. In addition, an important tool in our analysis is the use of pointwise weak regularity properties of vector-valued distributions and their (Tauberian) characterizations in terms of the wavelet transform \cite{vindas-pilipovic,vindas-p-r}. Even in the classical case $C_{\ast}^{\alpha}(\mathbb{R}^{n})$, our analysis provides a new approach to the study of H\"{o}lder-Zygmund spaces. It is then very likely that this kind of arguments might also be applied to study other types of smooth spaces, such as Besov type spaces. 

Our new classes of spaces $C^{\alpha,L}(\mathbb{R}^{n})$ and $C_{\ast}^{\alpha,L}(\mathbb{R}^{n})$, the $L$-H\"{o}lder and $L$-Zygmund spaces of exponent $\alpha$ that will be introduced in Section \ref{L-Zygmund}, contribute to the analysis of local regularity of functions or distributions by refining the regularity scale provided by the classical H\"{o}lder-Zygmund spaces. In fact, as explained in Section \ref{L-Zygmund}, they satisfy the useful inclusion relations:
$$
C^{\beta}(\mathbb{R}^{n})\subset
C^{\alpha,L}(\mathbb{R}^{n}) \subset C^{\gamma}(\mathbb{R}^{n}), \ \ \ \text{ when } 0 < \gamma < \alpha < \beta. $$
and 
$$
C^{\beta}_{\ast}(\mathbb{R}^{n})\subset
C^{\alpha,L}_{\ast}(\mathbb{R}^{n}) \subset C^{\gamma}_{\ast}(\mathbb{R}^{n}), \ \ \ \text{ if } \gamma < \alpha < \beta. $$
Situations in which these kinds of refinements are essential often occur in the literature and they have  already shown to be meaningful in applications. The particular instance $L(y)=|\log y|^{\beta}$ has been extensively studied (see, e.g., \cite[p. 276]{hol1}). Our analysis will treat more general weights; specifically, the important case when $L$ is a slowly varying function \cite{bingham,seneta}.

The paper is organized as follows. We review in Section \ref{snn} basic facts about test function spaces, the wavelet transform and its left-inverse, namely, the wavelet synthesis operator.
In Section \ref{continuityth}, we will provide the announced new proof of the continuity theorem for the wavelet and wavelet synthesis transforms when acting on test function spaces. We then explain in Section \ref{fn} some useful concepts that will be applied to the analysis of our weighted versions of the H\"{o}lder-Zygmund spaces; in particular, we shall discuss there the notion of (generalized) Littlewood-Paley pairs and some results concerning pointwise weak regularity of vector-valued distributions. Finally, we give the definition and study relevant properties of the new class of H\"{o}lder-Zygmund spaces in Section \ref{L-Zygmund}.

\section{Notation and notions}
\label{snn}
We denote by
$\mathbb{H}^{n + 1} = \mathbb{R}^n \times \mathbb{R}_+ $ the upper half-space. If $ x
\in \mathbb{R}^n $ and $ m \in \mathbb{N}^n $, then $|x|$ denotes the euclidean norm,
$x^{m} = x_1^{m_1} \dots x_n^{m_n},$
$\partial^{m}=\partial_x^{m} = \partial_{x_1}^{m_1} \dots\partial_{x_n}^{m_n},$
$m! = m_1! m_2!
\dots m_n! $ and $ |m| = m_1 + \dots + m_n. $
If the $ j$-th coordinate of $ m $ is one and the others vanish, we then write $
\partial_j=\partial_x^{m}.$ The set $B(0,r)$ is the euclidean ball in $\mathbb{R}^{n}$ of radius $r$. In the sequel, we use $C$ and $C'$ to denote positive
constants which may be different in various
occurrences.

\subsection{Function and distribution spaces}
The well known \cite{schwartz1} Schwartz space of rapidly decreasing smooth test functions is denoted by $\mathcal{S} (\mathbb{R}^{n})$. We will fix constants in the Fourier transform as
$\hat{\varphi}(\xi) = \int_{\mathbb{R}^{n}} \varphi(t) e^{-i\xi \cdot
t} \mathrm{d}t.$ The moments of
$\varphi\in\mathcal{S}(\mathbb{R}^{n})$ are denoted by $\mu_{m}(\varphi)=\int_{\mathbb{R}^{n}}t^{m}\varphi(t)\mathrm{d}t$,
$m\in\mathbb{N}^{n}$.

Following \cite{hol1}, the space of highly time-frequency localized functions $\mathcal{S}_0 (\mathbb{R}^{n})$ is defined as
$
\mathcal{S}_0 (\mathbb{R}^{n})=\left\{\varphi\in \mathcal{S}(\mathbb{R}^{n}): \: \mu_m (\varphi) = 0,\ \forall m \in \mathbb{N}^{n} \right\};
$
it is provided with the relative topology inhered from $\mathcal{S}(\mathbb{R}^{n}).$ In \cite{hol1}, the topology of $\mathcal{S}_{0}(\mathbb{R}^{n})$ is introduced in an apparently different way; however, both approaches are equivalent in view of the open mapping theorem.
Observe that $
\mathcal{S}_0 (\mathbb{R}^{n})$ is a closed subspace of $\mathcal{S}(\mathbb{R}^{n})$ and that $\varphi\in\mathcal{S}_{0} (\mathbb{R}^{n}) $ if and only if $ \partial^{m}\hat{\varphi}(0)= 0$ for all $m \in\mathbb{N}^{n}.$ It is important to point out that $
\mathcal{S}_0 (\mathbb{R}^{n})$ is also well known in the literature as the \emph{Lizorkin} space of test functions (cf. \cite{samko}).

The space $ \mathcal{S} (\mathbb {H}^{n+1}) $ of highly localized functions on the half-space \cite{hol1} consists of those $ \Phi \in C^{\infty} (\mathbb{H}^{n+1}) $ for which
$$\rho^0_{l,k,\nu,m}(\Phi)=\sup_{(x,y)\in
\mathbb {H}^{n+1}}\left(y^{l}+\frac
{1}{y^{l}}\right)\left(1+\left|x\right|^2\right)^{k/2}\,\left|\partial^{\nu}_{y}\partial_{x}^{m}\Phi
(x,y)\right|<\infty,$$ for all $l,k,\nu\in \mathbb{N}$
and $m\in\mathbb{N}^{n}$.
The canonical topology on $ \mathcal{S} (\mathbb {H}^{n+1}) $ is induced by this family of seminorms \cite{hol1}. For later use, we shall denote by $\rho_{k,m}$
the corresponding seminorms in $\mathcal{S}(\mathbb{R}^{n})$, namely,

$$\rho_{k,m}(\varphi)= \sup_{t\in
\mathbb {R}^{n}} \left(1+\left|t\right|^2\right)^{k/2}\,\left|\partial^{m}\varphi
(t)\right| ,\ \ \ k\in \mathbb{N},\ m\in \mathbb{N}^n.$$

The corresponding duals of these three spaces are $\mathcal{S}'(\mathbb{R}^{n}) $, $\mathcal{S}'_{0}(\mathbb{R}^{n}) $, and $\mathcal{S}'(\mathbb{H}^{n+1})$. They are, respectively, the spaces of tempered distributions, Lizorkin distributions, and distributions of slow growth on $\mathbb{H}^{n+1}$. Since the elements of $\mathcal{S}_{0}(\mathbb{R}^{n})$ are orthogonal to every polynomial, $\mathcal{S}'_{0}(\mathbb{R}^{n})$ can be canonically identified with the quotient space of $\mathcal{S}'(\mathbb{R}^{n})$ modulo polynomials.

Finally, we shall also make use of spaces of vector-valued tempered distributions \cite{silva,treves}. If $X$ is a locally convex topological vector space \cite{treves}, then the space of $X$-valued tempered distributions is $
\mathcal{S}'(\mathbb{R}^{n}, X) = L_{b}(\mathcal{S}
(\mathbb{R}^{n}), X)$, namely, the space of continuous linear mappings
from $\mathcal{S}(\mathbb{R}^{n})$ into $X$.

\subsection{Wavelet transform}
\label{subswt}

In this article a wavelet simply means a function $\psi\in\mathcal{S}(\mathbb{R}^{n})$ that satisfies
$\mu_{0}(\psi)=\int_{\mathbb{R}^{n}}\psi(t)\mathrm{d}t=0$.

The wavelet transform of
$ f \in
\mathcal{S}' (\mathbb{R}^{n})$ with respect to the wavelet $ \psi\in\mathcal{S}(\mathbb{R}^{n})$ is defined as
\begin{equation}
 \label{eq34} \mathcal{W}_{\psi} f (x, y) = \left\langle f(t), \frac{1}{y^n}
\bar{\psi}\left(\frac{t - x}{y}\right) \right\rangle = \frac{1}{y^n}\int_{\mathbb{R}^n} f(t)
 \bar{\psi}\left(\frac{t - x}{y}\right) \,\mathrm{d}t,
\end{equation}
where $(x,y)\in\mathbb{H}^{n+1}$. The very last integral formula is a formal notation which makes sense
when $f$ is a function of tempered growth. Notice that the wavelet transform is
also well defined via (\ref{eq34}) for $ f \in \mathcal{S}'_0
(\mathbb{R}^{n}) $ if the wavelet $ \psi \in
\mathcal{S}_0(\mathbb{R}^{n}).$ The wavelet transform can be defined exactly in the same way for vector-valued distributions.

\subsection{Wavelet synthesis operator}
\label{subswso}
Let $\eta \in \mathcal{S}_{0}(\mathbb{R}^{n})$. The wavelet synthesis transform of $\Phi\in \mathcal{S}(\mathbb{H}^{n+1})$ with respect to the wavelet $\eta$ is defined as
\begin{equation}
\label{wnwneq6} \mathcal{M}_{\eta} \Phi(t) = \int^{\infty}_{0}
\int_{\mathbb{R}^{n}} \Phi(x,y) \frac{1}{y^{n}} \eta
\left(\frac{t-x}{y} \right) \frac{\mathrm{d}x\mathrm{d}y}{y}\: , \
\ \ t \in \mathbb{R}^{n}.
\end{equation}
Observe that the operator $\mathcal{M}_{\eta}$ may be extended to act on the space $\mathcal{S}'(\mathbb{H}^{n+1})$ via duality arguments, see \cite{hol1} for details (cf. \cite{vindas-pilipovic} for the vector-valued case). In this paper we restrict our attention to its action on the test function space  $\mathcal{S}(\mathbb{H}^{n+1})$.

The importance of the wavelet synthesis operator lies in fact that it can be used to construct a left inverse for the wavelet transform, whenever the wavelet possesses nice reconstruction properties. Indeed, assume that $ \psi\in\mathcal{S}_{0} (\mathbb{R}^{n})$ admits a reconstruction wavelet $\eta\in\mathcal{S}_{0}(\mathbb{R}^{n})$. More precisely, it means that the constant
$$c_{\psi,\eta}=c_{\psi,\eta}(\omega) = \int^{\infty}_{0}
\overline{\hat{\psi}}(r\omega) \hat{\eta}(r\omega)
\frac{\mathrm{d}r}{r} \: , \ \ \ \omega\in\mathbb{S}^{n-1},
$$
is different from zero and independent of the direction $ \omega$. Then, a straightforward calculation \cite{hol1} shows that
\begin{equation}
\label{eqinv}
\mathrm{Id}_{\mathcal{S}_0(\mathbb{R}^{n})}=\frac{1}{c_{\psi,\eta}}
\mathcal{M}_{\eta}\mathcal{W}_{\psi}.
\end{equation}
It is worth pointing out that (\ref{eqinv}) is also valid \cite{hol1,vindas-pilipovic} when $\mathcal{W}_{\psi}$ and $\mathcal{M}_{\eta}$ act on the spaces $\mathcal{S}'_{0}(\mathbb{R}^{n})$ and $\mathcal{S}'(\mathbb{H}^{n+1})$, respectively.

Furthermore, it is very important to emphasize that a wavelet $\psi$ admits a reconstruction wavelet $\eta$ if and only if it is non-degenerate in the sense of the following definition \cite{vindas-pilipovic}:

\begin{definition}
\label{defnd} A test function
$\varphi\in\mathcal{S}(\mathbb{R}^{n})$ is said to be \emph{non-degenerate} if for
any $\omega\in\mathbb{S}^{n-1}$ the function of one variable
$R_{\omega}(r)=\hat{\varphi}(r\omega)\in C^{\infty}[0,\infty)$ is
not identically zero, that is, $ \operatorname*{supp}
R_{\omega}\neq\emptyset, $ for each $ \omega\in\mathbb{S}^{n-1}. $
\end{definition}

\section{The wavelet transform of test functions}
\label{continuityth}
The wavelet and wavelet synthesis transforms induce the bilinear mappings
\begin{equation*}
\mathcal{W}:(\psi,\varphi)\mapsto \mathcal{W}_{\psi}\varphi \ \ \ \text{and} \ \ \ \mathcal{M}:(\eta,\Phi)\mapsto \mathcal{M}_{\eta}\Phi.
\end{equation*}
Our first main result is a new proof of the continuity theorem for these two bilinear mappings when acting on test function spaces. Such a result was originally obtained by Holschneider \cite{hol93,hol1}. Our proof is elementary and significantly simpler than the one given in \cite{hol1}.

\begin{theorem} \label{th1} The two bilinear mappings
\begin{itemize}
\item [(i)] $\mathcal{W}:\mathcal{S}_0
(\mathbb{R}^n)\times \mathcal{S}_0
(\mathbb{R}^n)\rightarrow \mathcal{S} (\mathbb{H}^{n + 1})$
\item [(ii)] $ \mathcal{M}: \mathcal{S}_0
(\mathbb{R}^n)\times \mathcal{S} (\mathbb{H}^{n + 1})\rightarrow \mathcal{S}_0
(\mathbb{R}^n) $
\end{itemize}
are continuous.
\end{theorem}
\begin{proof} \emph{Continuity of the wavelet mapping.} We will prove that for arbitrary $l, k, \nu
\in \mathbb{N}, m \in \mathbb{N}^n $, there exist $N\in\mathbb N$ and $C>0$ such that
\begin{equation}\label{con1}
\rho^0_{l,k,\nu,m}\left(\mathcal{W}_{\psi}\varphi\right)\leq C \sum_{i', |j'|,i, |j| \leq N}
\rho_{i', j'} ({\psi})\rho_{i, j} ({\varphi}).
\end{equation}
We begin by making some reductions. Observe that, for constants $c_{j}$ which do not depend on $\varphi$ and $\psi$,
\begin{align*}
\partial_y^{\nu} \partial_x^{m}\mathcal{W}_{\psi}\varphi(x,y)&= \partial_y^{\nu} \partial_x^{m}
\int_{\mathbb{R}^{n}}\varphi(yt+x)\bar{\psi}(t)\mathrm{d}t
\\
&
= \sum_{|j| \leq \nu} c_{j}
\int_{\mathbb{R}^{n}}\partial^{m+j}\varphi(yt + x) t^{j} \bar{\psi} (t)\mathrm{d}t
\\
&
=
\sum_{|j| \leq \nu} c_{j}\mathcal{W}_{\psi_{j}}\varphi_{m+j}(x,y),
\end{align*}
where $\psi_{j}(t)=t^{j} \psi (t) \in \mathcal{S}_0 (\mathbb{R}^n) $ and
$ \varphi_{m+j}=\partial^{m+j}\varphi\in \mathcal{S}_0 (\mathbb{R}^n) $. It is therefore enough to
show (\ref{con1}) for $ \nu =0$ and  $m=0.$ Next, we may assume that $k$ is even.  We then have, for constants $c_{r,s}$ independent of $\psi$ and $\varphi$,
\begin{align*}
(1 + |x|^2)^{k/2} \mathcal{W}_{\psi}\varphi(x, y)&= \frac{1}{(2 \pi)^n}\left\langle (1 - \Delta_\xi)^{k/2} e^{i
\xi\cdot x},\hat{\varphi}(\xi)
\overline{\hat{\psi}}(y \xi) \right\rangle
\\
&
=\frac{1}{(2 \pi)^n} \sum_{|r| + |s| \leq k} c_{r, s}y^{|s|}
\int_{\mathbb{R}^n}  e^{i
\xi\cdot x}\partial^{r}\hat{\varphi}(\xi)
\partial_{\xi}^{s}\overline{\hat{\psi}}(y\xi) \, \mathrm{d}\xi
\\
&
=\frac{1}{(2 \pi)^n} \sum_{|r| + |s| \leq k} c_{r, s}y^{|s|}\mathcal{W}_{\psi_{s}}\varphi_{r}(x,y),
\end{align*}
where $\varphi_{r}(t)= (-it)^{r}\varphi(t)$ and $\psi_{s}(t)=(it)^{s}\psi(t)$. Thus, it clearly suffices to establish (\ref{con1}) for $k=\nu=0$ and $m=0$. We may also assume that $l\geq n$.

We first estimate the term $ y^l |\mathcal{W}_{\psi}\varphi(x,y)|. $ Since $\partial^{j}\hat{\varphi}(0) = 0$ for every $ j\in \mathbb{N}^n$, we can apply the Taylor formula to obtain
\begin{equation*}
\hat{\varphi}(\xi)=\sum_{\left|j\right|=l-n}\frac{\partial^{j}\hat{\varphi}(z_{\xi})}{j!}\:\xi^{j}\: , \ \ \ \text{for some }z_{\xi} \text{ in the line segment } [0,\xi].
\end{equation*}
Hence,
\begin{align*}
y^l |\mathcal{W}_{\psi}\varphi(x,y)|& = \frac{y^{l}}{(2 \pi)^n} \left|\left\langle e^{i \xi \cdot x}
\hat{\varphi}(\xi), \overline{\hat{\psi}}(y \xi) \right\rangle\right|
\leq \frac{1}{(2 \pi)^n} \int_{\mathbb{R}^{n}}
|\hat{\varphi}(\xi)| y^l |\hat{\psi}(y
\xi)| \, \mathrm{d}\xi
\\
&
\leq \frac{1}{(2 \pi)^n} \sum_{\left|j\right|=l-n}\frac{1}{j!} \int_{\mathbb{R}^{n}}|\partial^{j}\hat{\varphi}(z_{\xi})|
 y^n |(y\xi)^{j}\hat{\psi}(y
\xi)|\, \mathrm{d}\xi
\\
&
\leq \frac{1}{(2 \pi)^n} \sum_{\left|j\right|=l-n}\frac{\rho_{0,j}(\hat{\varphi})}{j!} \int_{\mathbb{R}^{n}}|\xi^{j}\hat{\psi}(
\xi)| \, \mathrm{d}\xi
\\
&
\leq \frac{1}{(2 \pi)^n} \sum_{\left|j\right|=l-n}\frac{\rho_{0,j}(\hat{\varphi})\rho_{l+n,0}(\hat{\psi})}{j!} \int_{\mathbb{R}^{n}}\frac{\mathrm{d}\xi}{\left(1+\left|\xi\right|^{2}\right)^{n}}
\\
&
\leq
C \rho_{l+n,0}(\varphi)\sum_{\left|j\right|\leq 2l+2n}\rho_{2n,j}(\psi), \ \ \ \text{for all }(x,y)\in\mathbb{H}^{n+1}.
\end{align*}

It remains to estimate $ y^{-l} |\mathcal{W}_{\psi}\varphi (x,y)| $. We shall now use the fact that all
the moments of $\psi$ vanish. If we apply the Taylor formula, we have, for some $z_{t}=z(t,x,y)$ in the line segment $[x,yt]$,
\begin{align*}
\frac{1}{y^l} | \mathcal{W}_{\psi}\varphi (x,y)|& = \frac{1}{y^l} \left|\int_{\mathbb{R}^{n}} \varphi(yt+x) \bar{\psi} (t) \mathrm{d}t\right|
\\
&
=\frac{1}{y^l}\left| \int_{\mathbb{R}^{n}} \bar{\psi} (t)\left(\sum_{|j| < l}\frac{\partial^{j}\varphi(x)}{j!} (y t)^{j} +\sum_{|j| = l} \frac{\partial^{j}\varphi(z_{t})}{j!} (yt)^{j}\right) \mathrm{d}t\right|
\\
&
\leq\sum_{\left|j\right|=l}\frac{\rho_{0,j}(\varphi)}{j!}\int_{\mathbb{R}^{n}}(1+\left|t\right|^{2})^{l/2}|\psi(t)|\mathrm{d}t
\\
&
\leq C \rho_{l+2n,0}(\psi)\sum_{\left|j\right|=l}\rho_{0,j}(\varphi).
\end{align*}
The result immediately follows on combining the previous two estimates.

\emph{Continuity of the wavelet synthesis mapping.}
We should now
prove that for arbitrary $ k \in \mathbb{N}$ and $\kappa\in
\mathbb{N}^n$
there exist $N\in\mathbb N$ and $C>0$
such that
\begin{equation}\label{con2}
\rho_{k,\kappa}(\mathcal{M}_{\eta}\Phi)\leq C\underset{i,\left|j\right|\leq N}{\sum_{k_1,k_2,l,|m|\leq N}}
\rho_{i,j}(\eta)\rho^0_{k_1,k_2,l,m}(\Phi).
\end{equation}
Since $\partial_{t}^{\kappa} \mathcal{M}_{\eta}\Phi=\mathcal{M}_{\eta}\partial_{x}^{\kappa}\Phi$, it is enough to prove (\ref{con2}) for $ \kappa=0.$ We denote below $\hat{\Phi}$ the partial Fourier transform of $\Phi$ with respect to the space coordinate, i.e., $\hat{\Phi}(\xi,y)=\int_{\mathbb{R}^{n}}\Phi(x,y)e^{-i \xi\cdot x}\mathrm{d}x$. We may assume that $k$ is even.
We then have,
\begin{align*}
&(1 + |t|^2)^{\frac{k}{2}} |\mathcal{M}_{\eta} \Phi(t)|= (1 + |t|^2)^{\frac{k}{2}} \left|\int_{\mathbb{R}_+}
\int_{\mathbb{R}^n} \Phi(t - x, y) \, \frac{1}{y^n}\eta\left(\frac{x}{y}\right) \, \frac{\mathrm{d}x \mathrm{d}y}{y}\right|
\\
&
= \frac{(1 + |t|^2)^{\frac{k}{2}}}{(2 \pi)^n}
\left|\int_{\mathbb{R}_+} \int_{\mathbb{R}^n} \frac{(1 -
\Delta_{\xi})^{\frac{k}{2}} e^{i t \xi}}{(1 +
|t|^2)^{\frac{k}{2}}} \, \hat{\Phi}(\xi, y) \, \hat{\eta}(y
\xi) \, \frac{\mathrm{d}\xi \mathrm{d}y}{y}\right|
\\
&
\leq \frac{1}{(2\pi)^{n}} \int_{\mathbb{R}_+} \int_{\mathbb{R}^n} \left|(1 -
\Delta_{\xi})^{\frac{k}{2}} \left(\hat{\Phi}(\xi, y) \,
\hat{\eta}(y \xi)\right)\right| \, \frac{\mathrm{d}\xi \mathrm{d}y}{y}
\\
&
\leq \frac{1}{(2\pi)^{n}} \sum_{|r| + |s|
\leq k} c_{r, s}\int_{\mathbb{R}_+} \int_{\mathbb{R}^n}  |\partial^r_{\xi} \hat{\Phi}(\xi, y)| \, y^{|s|-1}
|\partial^{s}\hat{\eta} (y \xi)| \, \mathrm{d}\xi \mathrm{d}y
\\
&
\leq \frac{1}{(2\pi)^{n}} \sum_{|r| + |s|
\leq k} c_{r, s}\rho_{0,s}(\hat{\eta})\int_{\mathbb{R}_+} \int_{\mathbb{R}^n} y^{|s|-1} |\partial^r_{\xi} \hat{\Phi}(\xi, y)| \, \mathrm{d}\xi \mathrm{d}y
\\
&
\leq C'\sum_{|r| + |s|
\leq k}\rho_{0,s}(\hat{\eta})(\rho^{0}_{\left|s\right|-1, 2n,0 ,r}(\hat{\Phi})+\rho^{0}_{\left|s\right|+1, 2n,0 ,r}(\hat{\Phi}))\int_{\mathbb{R}_+} \int_{\mathbb{R}^n}  \frac{\mathrm{d}\xi}{(1 + |\xi|^2)^n}\, \frac{\mathrm{d}y}{1 + y^2}
\\
&
\leq C\underset{\left|j\right|\leq 2n}{\sum_{|r| + |s|
\leq k}}\rho_{\left|s\right|+2n,0}(\eta)(\rho^{0}_{\left|s\right|-1, |r|+2n,0 ,j}(\Phi)+\rho^{0}_{\left|s\right|+1, |r|+2n,0 ,j}(\Phi)).
\end{align*}
This completes the proof.\end{proof}
\begin{remark} It follows from the proof of the continuity of $\mathcal{M}$ that we can extend the bilinear mapping $\mathcal{M}:(\eta,\Phi)\mapsto \mathcal{M}_{\eta}\Phi$ to act on
$$\mathcal{M}:\mathcal{S}(\mathbb{R}^{n})\times \mathcal{S}(\mathbb{H}^{n+1})\rightarrow \mathcal{S}(\mathbb{R}^{n}),$$
 and it is still continuous.
\end{remark}
\section{Further notions}
\label{fn}
Our next task is to define and study the properties of a new class of weighted H\"{o}lder-Zygmund spaces. We postpone that for Section \ref{L-Zygmund}. In this section we collect some useful concepts that will play an important role in the next section.

\subsection{Generalized Littlewood-Paley pairs}
\label{LPpairs}
In our definition of weighted Zygmund spaces we shall employ a generalized Littlewood-Paley pair \cite{vindasLP}. They generalize those occurring in familiar (continuous) Littlewood-Paley decompositions of the unity (cf. Example \ref{ex1} below).

Let us start by introducing the index of non-degenerateness of wavelets, as defined in \cite{vindas-pilipovic}. Even if a wavelet is non-degenerate, in the sense of Definition \ref{defnd}, there may be a ball on which its Fourier transform ``degenerates''. We measure in the next definition how big that ball is.

\begin{definition}
\label{defind} Let $\psi\in\mathcal{S}(\mathbb{R}^{n})$ be a non-degenerate wavelet. Its \emph{index of non-degenerateness} is the (finite) number
$$
\tau=\inf\left\{r\in\mathbb{R}_{+}:\:\operatorname*{supp}
R_{\omega}\cap[0,r]\neq \emptyset,
\forall\omega\in\mathbb{S}^{n-1}\right\},
$$
where $R_{\omega}$ are the functions of one variable $R_{\omega}(r)=\hat{\psi}(r\omega)$.
\end{definition}

If we only know values of $\mathcal{W}_{\psi}f(x,y)$ at scale $y<1$, then the wavelet transform can be blind when analyzing certain distributions (cf. \cite[Sec. 7.2]{vindas-pilipovic}. The idea behind the introduction of Littlewood-Paley pairs is to have an alternative way for recovering such a possible lost of information by employing additional data with respect to another function $\phi$ (cf. \cite{vindasLP}).

\begin{definition} \label{def5} Let $ \alpha\in\mathbb{R}$, $\phi \in\mathcal{S}(\mathbb{R}^{n})$. Let $ \psi \in {\mathcal S} (\mathbb{R}^n) $ be a non-degenerate
wavelet with the index of non-degenerateness $ \tau.$ The pair $(\phi,\psi)$ is said to be a \emph{Littlewood-Paley pair} (LP-pair) of order $\alpha$ if  $ \hat{\phi} (\xi) \neq 0 $ for $|\xi|\leq \tau $ and
$\mu_{m} (\psi) = 0$ for all multi-index $m\in\mathbb{N}^{n}$ with $\left| m \right|\leq [\alpha] $.
\end{definition}
\begin{example} \label{ex1} Let $ \phi \in {\mathcal S}
(\mathbb{R}^n) $ be a radial function such that $ \hat{\phi} $ is nonnegative, $
\hat{\phi} (\xi) = 1$ for $|\xi| < 1/2$ and $\hat{\phi}(\xi)=0$ for $|\xi| >1$. Set $ \hat{\psi} (\xi) = -
\xi\cdot \nabla\hat{\varphi}(\xi).$ The pair $(\phi, \psi) $ is then clearly a LP-pair of order $\infty$. Observe that this well known pair is the one used in the so-called Littlewood-Paley decompositions of the unity and plays a crucial role in the study of various function spaces, such as the classical Zygmund space $ C_{\ast}^{\alpha}(\mathbb{R}^{n})$ (cf., e.g., \cite{horm}).
\end{example}

We pointed out above that LP-pairs enjoy powerful reconstruction properties. Let us make this more precise.

\begin{proposition}
\label{pLP} Let $(\phi,\psi)$ be a LP-pair, the wavelet $\psi$ having index of non-degenerateness $\tau$ and $r>\tau$ being a number such that $\hat{\phi}(\xi)\neq 0$ for $\left|\xi\right|<r$. Pick any $\sigma$ in between $\tau$ and $r$. If $\eta\in\mathcal{S}_{0}(\mathbb{R}^{n})$ is a reconstruction wavelet for $\psi$ whose Fourier transform has support in $B(0,\sigma)$ and $\varphi\in\mathcal{D}(\mathbb{R}^{n})$ is such that $\varphi(\xi)=1$ for $\xi\in B(0,\sigma)$ and $\operatorname*{supp}\varphi\subset B(0,r)$, then, for all $f\in\mathcal{S}'(\mathbb{R}^{n})$ and $\theta\in\mathcal{S}(\mathbb{R}^{n})$
\begin{equation}
\label{eqLP}
\left\langle f,\theta \right\rangle= \int_{\mathbb{R}^{n}}(f\ast \phi)(t) \theta_{1}(t)\mathrm{d}t\: + \:\frac{1}{c_{\psi,\eta}}\int_{0}^{1}\int_{\mathbb{R}^{n}} \mathcal{W}_{\psi}f(x,y)\mathcal{W}_{\bar{\eta}}\theta_{2}(x,y)\frac{\mathrm{d}x\mathrm{d}y}{y}
,\end{equation}
where $\hat{\theta}_{1}(\xi)=\hat{\theta}(\xi)\varphi(\xi)/\hat{\phi}(-\xi)$ and $\hat{\theta}_{2}(\xi)=\hat{\theta}(\xi)(1-\varphi(\xi))$.
\end{proposition}

\begin{proof}
Observe that
$$
\left\langle f\ast\phi,\theta_{1}\right\rangle= (2\pi)^{-n}\langle \hat{f}(\xi),\hat{\theta}(-\xi)\varphi(-\xi)\rangle.
$$
It is therefore enough to assume that $\theta_{1}=0$ so that $\theta=\theta_{2}$. Our assumption over $\eta$ is that $\eta\in\mathcal{S}_{0}(\mathbb{R}^{n})$, $\operatorname*{supp} \hat{\eta}\subset B(0,\sigma),$ and
$$
c_{\psi,\eta}=\int_{0}^{\infty}\overline{\hat{\psi}}(r\omega)\hat{\eta}(r\omega)\frac{\mathrm{dr}}{r}\neq 0
$$
does not depend on the direction $\omega$. We remark that such a
reconstruction wavelet can always be found (see the proof of
\cite[Thrm. 7.7]{vindas-pilipovic}). Therefore,
$\mathcal{W}_{\bar{\eta}}\theta(x,y)=0$ for all
$(x,y)\in\mathbb{R}^{n}\times(1,\infty)$. Exactly as in \cite[p.
66]{hol1}, the usual calculation shows that
$$
\theta(t)=\frac{1}{c_{\psi,\eta}}\mathcal{M}_{\bar{\psi}}(\mathcal{W}_{\bar{\eta}}\theta)(t)=\frac{1}{c_{\psi,\eta}}\int_{0}^{1}\int_{\mathbb{R}^{n}}\bar{\psi}\left(\frac{t-x}{y}\right) \mathcal{W}_{\bar{\eta}}\theta(x,y)\frac{\mathrm{d}x\mathrm{d}y}{y}.
$$
Furthermore, since $\mathcal{W}_{\bar{\eta}}\theta\in \mathcal{S}(\mathbb{H}^{n+1})$ (cf. Theorem \ref{th1}), the last integral can be expressed as the limit in $\mathcal{S}(\mathbb{R}^{n})$ of Riemann sums. That justifies the exchange of dual pairing and integral in
\begin{align*}
\left\langle f,\theta \right\rangle&=\left\langle f(t),\frac{1}{c_{\psi,\eta}}\int_{0}^{1}\int_{\mathbb{R}^{n}}\bar{\psi}\left(\frac{t-x}{y}\right) \mathcal{W}_{\bar{\eta}}\theta(x,y)\frac{\mathrm{d}x\mathrm{d}y}{y}\right\rangle
\\
&
=\frac{1}{c_{\psi,\eta}}\int_{0}^{1}\int_{\mathbb{R}^{n}} \mathcal{W}_{\psi}f(x,y)\mathcal{W}_{\bar{\eta}}\theta(x,y)\frac{\mathrm{d}x\mathrm{d}y}{y}
\end{align*}
\end{proof}

\subsection{Slowly varying functions}
\label{rvf} The weights in our weighted versions of H\"{o}lder-Zygmund spaces will be taken from the class of Karamata regularly varying functions. Such functions have been very much studied and have numerous applications in diverse areas of mathematics. We refer to \cite{bingham, seneta} for their properties. Let us recall that a positive measurable function $L$ is called \emph{slowly varying} (at the origin) if it is asymptotically invariant under rescaling, that is,
\begin{equation}
\label{eqsv}
\lim_{\varepsilon\to0^{+}}\frac{L(a\varepsilon)}{L(\varepsilon)}=1, \ \ \ \text{for each }a>0.
\end{equation}
Familiar functions such as  $1$, $\left|\log \varepsilon\right|^{\beta},$ $\left|\log\left|\log \varepsilon\right|\right|^{\beta}$, ..., are slowly varying. Regularly varying functions are then those that can be written as $\varepsilon^{\alpha}L(\varepsilon)$, where $L$ is slowly varying and $\alpha\in\mathbb{R}$.
\subsection{Weak-asymptotics}
\label{wa}
We shall use some notions from the theory of asymptotics of generalized functions \cite{estrada-kanwal2,p-s-v,vindas-pilipovic,vladimirov-d-z1}. The weak-asymptotics of distributions, also known as quasi-asymptotics, measure pointwise scaling growth of distributions with respect to regularly varying functions in the weak sense. Let $E$ be a Banach space with norm $\left\| \ \right\|$ and let $L$ be slowly varying. For $\mathbf{f}\in \mathcal{S}'(\mathbb{R}^{n},E)$, we write
\begin{equation*}
\mathbf{f}\left(x_0+\varepsilon
t\right)=O\left(\varepsilon^{\alpha}L(\varepsilon)\right)\ \ \
\mbox{as}\  \varepsilon\to0^{+}  \ \ \mbox{in}\
\mathcal{S}'(\mathbb{R}^{n},E),
\end{equation*}
if the order growth relation holds after evaluation at each test function, i.e.,
\begin{equation}\label{eqasymp}
\left\|\left\langle \mathbf{f}\left(x_0+\varepsilon
t\right),\varphi(t)\right\rangle\right\|\leq C_{\varphi}\varepsilon^{\alpha}L(\varepsilon), \ \ \ 0<\varepsilon\leq 1,
\end{equation}
for each test function $\varphi\in\mathcal{S}(\mathbb{R}^{n}).$ Observe that weak-asymptotics are directly involved in Meyer's notion of the scaling weak pointwise exponent, so useful in the study of pointwise regularity and oscillating properties of functions \cite{meyer}.

One can also use these ideas to study exact pointwise scaling asymptotic properties of distributions (cf. \cite{estrada-kanwal2,p-s-v,vindas-pilipovic,vindas-pilipovic2009}). We restrict our attention here to the important notion of the value of a distribution at a point, introduced and studied by \L ojasiewicz in \cite{lojasiewicz,lojasiewicz2} (see also
\cite{estrada-vindasI,vindas-estrada1,walterw}). The vector-valued distribution $\mathbf{f}\in\mathcal{S}'(\mathbb{R}^{n},E)$ is said to have a value $\mathbf{v}\in E$ at the point $x_{0}\in\mathbb{R}^{n}$ if $\lim_{\varepsilon\to0^{+}}\mathbf{f}(x_{0}+\varepsilon t)=\mathbf{v}$, distributionally, i.e., for each $\varphi\in\mathcal{S}(\mathbb{R}^{n})$
\begin{equation*}
\lim_{\varepsilon\to0^{+}}\left\langle
\mathbf{f}\left(t\right),\frac{1}{\varepsilon^{n}}\varphi\left(\frac{t-x_{0}}{\varepsilon}\right)\right\rangle=
\mathbf{v}\int_{\mathbb{R}^{n}}\varphi(t)\mathrm{d}t \in E.
\end{equation*}
In such a case we simply write $\mathbf{f}(x_{0})=\mathbf{v}$, distributionally.

\subsection{Pointwise weak H\"{o}lder space}
\label{sphs} An important tool in Section \ref{L-Zygmund} will be the
concept of pointwise weak H\"{o}lder spaces of vector-valued distributions and their intimate connection with boundary asymptotics of the wavelet transform.  These pointwise spaces have been recently introduced and investigated in \cite{vindas-pilipovic}. They are extended versions of Meyer's pointwise weak spaces from \cite{meyer}. They are also close relatives of Bony's two-microlocal spaces \cite{jaffard-meyer,meyer}. Again, we denote by $E$ a Banach space, $L$ is a
slowly varying function at the origin.

For a given $x_{0}\in\mathbb{R}^{n}$ and
$\alpha\in\mathbb{R}$, the pointwise weak H\"{o}lder space \cite{vindas-pilipovic} $C_{w}^{\alpha,L}(x_{0},E)$ consists of those distributions $\mathbf{f}\in \mathcal{S}'(\mathbb{R}^{n},E) $ for which there is an $E$-valued polynomial $\mathbf{P}$ of degree less than $\alpha$ such that (cf. Subsection \ref{wa})
\begin{equation}
\label{pseq}
\mathbf{f}(x_{0}+\varepsilon t)=\mathbf{P}(\varepsilon t)+O(\varepsilon^{\alpha}L(\varepsilon)) \ \ \ \mbox{as } \varepsilon\to0^{+} \mbox{ in }  \mathcal{S}'(\mathbb{R}^{n},E).
\end{equation}
Observe that if $\alpha<0$, then the polynomial is irrelevant. In addition, when $\alpha\geq0$, this polynomial is unique; in fact (\ref{pseq}) readily implies that the \L ojasiewicz point values $\partial^{m}\mathbf{f}(x_{0})$ exist, distributionally, for $\left|m\right|<\alpha$ and that $\mathbf{P}$ is the ``Taylor polynomial''
$$
\mathbf{P}(t)=\sum_{\left|m\right| <\alpha} \frac{\partial^{m}\mathbf{f}(x_{0})}{m!}\: t^{m}.
$$

The pointwise weak H\"{o}lder space $C_{\ast,w}^{\alpha,L}(x_{0},E)$ of second type is defined as follows: $\mathbf{f}\in C_{\ast,w}^{\alpha,L}(x_{0},E)$ if (\ref{eqasymp}) is just assumed to hold for each $\varphi\in\mathcal{S}(\mathbb{R}^{n})$ satisfying the requirement $\mu_{m}(\varphi)=0$ for all multi-index $\left|m\right|\leq \alpha$. Naturally, the previous requirement is empty if $\alpha<0$, thus, in such a case, $C_{\ast,w}^{\alpha,L}(x_{0},E)=C_{w}^{\alpha,L}(x_{0},E)$. One can also show that if $\alpha\notin\mathbb{N}$, the equality between these two spaces remains true \cite{vindas-pilipovic}. On the other hand, when $\alpha\in\mathbb{N}$, we have the strict inclusion $C_{w}^{\alpha,L}(x_{0},E)\subsetneq C_{\ast,w}^{\alpha,L}(x_{0},E)$ (cf. comments below). The usefulness of  $C_{\ast,w}^{\alpha,L}(x_{0},E)$ lies in the fact that it admits a precise wavelet characterization. The following theorem is shown in \cite{vindas-pilipovic}, it forms part of more general  Tauberian-type results that will not be discussed here.

\begin{theorem}[\cite{vindas-pilipovic}]
\label{tauberian}  Let $\alpha\in\mathbb{R}$ and let $\psi\in\mathcal{S}(\mathbb{R}^{n})$ be a non-degenerate
wavelet with $\mu_{m}(\psi)=0$ for
$\left|m\right|\leq[\alpha]$. Then, $\mathbf{f}\in
C_{\ast,w}^{\alpha,L}(x_{0},E)$ if and only if there is
$k\in\mathbb{N}$ such that
\begin{equation*}
\limsup_{\varepsilon\rightarrow0^+}\sup_{\left|x\right|^2+y^2=1,\:y>0}
\frac{y^k}{\varepsilon^{\alpha}L(\varepsilon)}\left\|\mathcal{W}_{\psi}
\mathbf{f}\left(x_0+\varepsilon x,\varepsilon
y\right)\right\|<\infty.
\end{equation*}

\end{theorem}

It is worth mentioning that the elements of $C_{\ast,w}^{\alpha,L}(x_{0},E)$ for $\alpha=p\in\mathbb{N}$ can be characterized by pointwise weak-asymptotic expansions. We have \cite{vindas-pilipovic} that $\mathbf{f}\in C_{\ast,w}^{p,L}(x_{0},E)$ if and only if it admits the weak-expansion
$$ \mathbf{f}(x_{0}+\varepsilon t)=\sum_{\left|m\right|< p} \frac{\partial^{m}\mathbf{f}(x_{0})}{m!}\: (\varepsilon t)^{m}+\varepsilon^{p}\sum_{\left|m\right|=p}t^{m}\mathbf{c}_{m}
(\varepsilon)+O\left(\varepsilon^{p}L(\varepsilon)\right) \ \mbox{in }\mathcal{S}'(\mathbb{R}^{n},E),
$$
where $\partial^{m}\mathbf{f}(x_{0})$ are interpreted in the \L ojasiewicz sense and the $\mathbf{c}_{m}:(0,\infty)\to E$ are continuous functions. Comparison of this weak-expansion with (\ref{pseq}) explains the difference between the two pointwise spaces when $\alpha=p\in\mathbb{N}$.

\section{New class of H\"{o}lder-Zygmund spaces}
\label{L-Zygmund}
Throughout this section, we assume that $L$ is a slowly varying function such that $L$ and $1/L$ are locally bounded on $(0,1]$.

\subsection{$L$-H\"{o}lder spaces} We now introduce weighted H\"{o}lder spaces with respect to $L$. They were already defined and studied in \cite{vindas-pilipovic}. Let
$\alpha\in\mathbb{R}_{+}\setminus\mathbb{N}$. We say that a function $f$
belongs to the space $C^{\alpha,L}(\mathbb{R}^{n})$ if $f$ has continuous derivatives up to order less than $\alpha$ and
\begin{equation}
\label{eqnH}
\left\|f\right\|_{C^{\alpha,L}}:=\sum_{\left|j\right|\leq [\alpha]}
\sup_{t\in\mathbb{R}^{n}}\left|\partial^{j}f(t)\right|+
\sum_{\left|m\right|=[\alpha]}\sup_{0<\left|t-x\right|\leq 1}
\frac{\left|\partial^{m}f(t)-\partial^{m} f(x)\right|}{\left|t-x\right|^{\alpha
-[\alpha]}L(\left|t-x\right|)}<\infty.
\end{equation}
When $\alpha=p+1\in\mathbb{N}$, we replace the previous requirement by
$$
\left\|f\right\|_{C^{\alpha,L}}:=\sum_{\left|j\right|\leq p}
\sup_{t\in\mathbb{R}^{n}}\left|\partial^{j}f(t)\right|+
\sum_{\left|m\right|=p}\sup_{0<\left|t-x\right|\leq 1}
\frac{\left|\partial^{m}f(t)-\partial^{m} f(x)\right|}{\left|t-x\right|L(\left|t-x\right|)}<\infty.
$$

The space $C^{\alpha,L}(\mathbb{R}^{n})$ is clearly a Banach space with the above norm. The conditions imposed over $L$ ensure that $C^{\alpha,L}(\mathbb{R}^{n})$ depends only on the behavior of $L$
near 0; thus, it is invariant under dilations. When $L\equiv1$,
this space reduces to
$C^{\alpha,L}(\mathbb{R}^{n})=C^{\alpha}(\mathbb{R}^{n})$, the
usual global (inhomogeneous) H\"{o}lder space \cite{horm, jaffard-meyer, meyer}.
Consequently, as in \cite{vindas-pilipovic}, we call
$C^{\alpha,L}(\mathbb{R}^{n})$ the global H\"{o}lder space with
respect to $L$. Note that, because of the properties of $L$ \cite{bingham,seneta}, we have the following interesting inclusion relations:
$$
 C^{\beta}(\mathbb{R}^{n})\subsetneq
C^{\alpha,L}(\mathbb{R}^{n}) \subsetneq C^{\gamma}(\mathbb{R}^{n}), \ \ \ \text{
whenever } 0 < \gamma < \alpha < \beta. $$

\subsection{$L$-Zygmund  spaces}
We now proceed to define the weighted Zygmund space $C^{\alpha, L}_{\ast}
(\mathbb{R}^n) $. Let  $\alpha\in\mathbb R$ and fix a LP-pair $ (\phi, \psi) $ of order $\alpha$. A distribution  $f \in\mathcal{S}' (\mathbb{R}^n)$ is said to belong to the $L$-Zygmund space $C^{\alpha, L}_{\ast}
(\mathbb{R}^n) $ of exponent $\alpha$  if
\begin{equation} \label{eq7}
\left\|f\right\|_{C^{\alpha,L}_{\ast}}:=  \|f \ast \phi \|_{L^{\infty}} + \sup_{x \in
\mathbb{R}^n} \sup_{0 < y \leq 1} \frac{1}{y^{\alpha} L(y)}
|{\mathcal W}_{\psi} f (x, y)| < \infty.
\end{equation}

Observe that we clearly have $C^{\alpha, L}
(\mathbb{R}^n) \subseteq C^{\alpha, L}_{\ast}
(\mathbb{R}^n) $, for $\alpha>0$. We will show that if $\alpha\in\mathbb{R}_{+}\setminus{\mathbb{N}}$, we
actually have the equality $C^{\alpha, L}
(\mathbb{R}^n) = C^{\alpha, L}_{\ast}
(\mathbb{R}^n) $. When $\alpha$ is a positive integer, we have in turn $C^{\alpha, L}
(\mathbb{R}^n) \subsetneq C^{\alpha, L}_{\ast}
(\mathbb{R}^n) $. As in the case of $L$-H\"{o}lder spaces, our $L$-Zygmund spaces refine the scale of classical Zygmund spaces; more precisely, we have again the inclusions:
$$
C^{\beta}_{\ast}(\mathbb{R}^{n})\subsetneq
C^{\alpha,L}_{\ast}(\mathbb{R}^{n}) \subsetneq C^{\gamma}_{\ast}(\mathbb{R}^{n}), \ \ \ \text{whenever } \gamma < \alpha < \beta. $$
The definition of $C^{\alpha, L}_{\ast}
(\mathbb{R}^n) $ can give the impression that it might depend on the choice of the LP-pair; however, this is not the case, as shown by the ensuing result.

\begin{proposition} \label{propindZ}
The definition of $C^{\alpha, L}_{\ast}
(\mathbb{R}^n) $ does not depend on the choice of the LP-pair. Moreover, different LP-pairs lead to equivalent norms.
\end{proposition}

In view of Proposition \ref{propindZ}, we may employ a LP-pair coming from a continuous Littlewood-Paley decomposition of the unity (cf. Example \ref{ex1}) in the definition of $C^{\alpha,L}_{\ast}(\mathbb{R}^{n})$. Therefore, when $ L \equiv 1,$ we recover the
classical Zygmund space $C^{\alpha,L}_{\ast}(\mathbb{R}^{n})=C^{\alpha}_{\ast}(\mathbb{R}^{n}) $ \cite{horm}. Proposition \ref{propindZ} follows at once from the following lemma.

\begin{lemma}\label{zl1}
Let $f\in  C^{\alpha, L}_{\ast} (\mathbb{R}^n),$
then for every $\theta\in  {\mathcal S}(\mathbb{R}^n)$ there holds
\begin{equation} \label{fp}
||f*\theta||_{L^{\infty}}\leq C \left\|f\right\|_{C^{\alpha,L}_{\ast}},
\end{equation}
where $\left\|f\right\|_{C^{\alpha,L}_{\ast}}$ is given by $(\ref{eq7})$. Furthermore,  if $\mathfrak{B}\subset \mathcal{S}(\mathbb{R}^{n})$ is a bounded set such that $\mu_m(\theta)=0$ for all $\theta\in\mathfrak{B}$ and all multi-index $m\leq [\alpha],$
then
\begin{equation} \label{sp}
\sup_{x\in\mathbb{R}^n}\sup_{0<y\leq 1} \frac{1}{y^\alpha L(y)}|\mathcal
W_\theta f(x,y)|\leq C \left\|f\right\|_{C^{\alpha,L}_{\ast}}, \ \ \ \mbox{for all }\theta\in\mathfrak{B}.
\end{equation}
\end{lemma}

\begin{proof}
The estimate (\ref{fp}) easily follows from the representation (\ref{eqLP}) of $f$ from Proposition \ref{pLP}. Let us show (\ref{sp}). We retain the notation from the statement of Proposition \ref{pLP}. In view of (\ref{eqLP}), a quick calculation yields
$$
\mathcal{
W}_\theta f(x,y)= F_{1}(x,y)+F_{2}(x,y)+F_{3}(x,y),
$$
where
$$
F_{1}(x,y)= \frac{1}{(2\pi)^{n}}\int_{\mathbb{R}^{n}}(f\ast\hat{\varphi})(x+yt)\bar{\theta}(t)\mathrm{d}t,
$$
$$
F_{2}(x,y)=\frac{1}{c_{\psi,\eta}}\int_{0}^{1/y}\int_{\mathbb{R}^{n}}\mathcal{W}_{\psi}f(yb+x,ya)\mathcal{W}_{\bar{\eta}}\theta_{2,y}(b,a)\frac{\mathrm{d}b\mathrm{d}a}{a},
$$
and
$$
F_{3}(x,y)=\frac{1}{c_{\psi,\eta}}\int_{0}^{1}\int_{\mathbb{R}^{n}}\mathcal{W}_{\psi}f(yb+x,ya)\mathcal{W}_{\bar{\eta}}\theta_{3,y}(b,a)\frac{\mathrm{d}b\mathrm{d}a}{a},
$$
with $\hat{\theta}_{2,y}(\xi)=\hat{\bar{\theta}}(\xi)(1-\varphi(\xi))(1-\varphi(\xi/y))$ and $\hat{\theta}_{3,y}(\xi)=\hat{\bar{\theta}}(\xi)\varphi(\xi)(1-\varphi(\xi/y))$.
To estimate $F_{1}$, we first observe that if $\alpha<0$, then
$$\sup_{0<y\leq 1}\sup_{x\in\mathbb{R}^{n}}\frac{\left|F_{1}(x,y)\right|}{y^\alpha L(y)}\leq C\left|f\right|_{C^{\alpha,L}_{\ast}}\left\|\theta\right\|_{L^{1}}\sup_{0<y\leq 1}\frac{1}{y^\alpha L(y)}\leq C'\left|f\right|_{C^{\alpha,L}_{\ast}},$$
because slowly varying functions satisfy the estimates $y^{\varepsilon}<CL(y)$ for any exponent $\varepsilon>0$ \cite{bingham,seneta}. When $\alpha\geq 0$, we have that $(f\ast\hat{\varphi})$ is a $C^{\infty}$-function with bounded derivatives of any order. Then, by the Taylor formula, (\ref{fp}) (with $\theta=\partial^{m}\hat{\varphi}$), and the assumption $\mu_{m}(\theta)=0$ for $\left|m\right|\leq [\alpha]$, we obtain
\begin{align*}
\sup_{0<y\leq 1}\sup_{x\in\mathbb{R}^{n}}\frac{\left|F_{1}(x,y)\right|}{y^{\alpha}L(y)}&\leq  \sup_{0<y\leq 1}C\left\|f\right\|_{C^{\alpha,L}_{\ast}} \frac{y^{[\alpha]+1-\alpha}}{L(y)}\int_{\mathbb{R}^{n}}\left|t\right|^{[\alpha]+1}\left|\theta(t)\right|\mathrm{d}t
\\
&
\leq C' \left\|f\right\|_{C^{\alpha,L}_{\ast}}.
\end{align*}
We now bound $F_{2}$ and $F_{3}$. If $\varepsilon>0$, Potter's estimate \cite[p. 25]{bingham} gives the existence of a constant $C=C_{\varepsilon}$ such that
$$
\frac{L(ay)}{L(y)}< C \left(a^{\varepsilon}+\frac{1}{a^{\varepsilon}}\right), \ \ \ \mbox{for all } 0<y<1 \mbox{ and }a<1/y.
$$
Thus,
\begin{align*}
\frac{\left|F_{2}(x,y)\right|}{y^{\alpha}L(y)}&\leq \frac{\left\|f\right\|_{C^{\alpha,L}_{\ast}}}{c_{\psi,\eta}}\int_{0}^{1/y}\int_{\mathbb{R}^{n}}a^{\alpha-1}\frac{L(ay)}{L(y)}\left|\mathcal{W}_{\bar{\eta}}\theta_{2,y}(b,a)\right|\mathrm{d}b\mathrm{d}a
\\
&
\leq C'\left\|f\right\|_{C^{\alpha,L}_{\ast}}\int_{0}^{\infty}\int_{\mathbb{R}^{n}}(a^{\alpha}+a^{\alpha-2})\left|\mathcal{W}_{\bar{\eta}}\theta_{2,y}(b,a)\right|\mathrm{d}b\mathrm{d}a.
\end{align*}
Notice that $\left\{\theta_{2,y}\in\mathcal{S}_{0}(\mathbb{R}^{n}): \:\theta\in\mathfrak{B},\: y\in(0,1] \right\}$ is a bounded set in $\mathcal{S}_{0}(\mathbb{R}^{n})$ because the derivatives of $\varphi$ are supported in $\left\{\xi: \: \sigma\leq\left|\xi\right|\leq r\right\}$. Thus, due to the continuity of $\mathcal{W}_{\bar{\eta}}$ (cf. Theorem \ref{th1}), $\left\{\mathcal{W}_{\bar{\eta}}\theta_{y}:\: \theta\in\mathfrak{B},\: y\in(0,1]\right\}$ is a bounded set in $\mathcal{S}(\mathbb{H}^{n+1})$. This implies that the integrals involved in the very last estimate are uniformly bounded for $\theta\in\mathfrak{B}$ and $y\in(0,1]$. 
Consequently, we obtain that

$$\sup_{0<y\leq 1}\sup_{x\in\mathbb{R}^{n}}\frac{\left|F_{2}(x,y)\right|}{y^{\alpha}L(y)}\leq C \left\|f\right\|_{C^{\alpha,L}_{\ast}}.$$
Next, for $F_{3}$, we have
$$ 
\frac{\left|F_{3}(x,y)\right|}{y^{\alpha}L(y)}\leq
C'\left\|f\right\|_{C^{\alpha,L}_{\ast}}\int_{0}^{1}\int_{\mathbb{R}^{n}}(a^{\alpha-3/2}+a^{\alpha-1/2})\left|\mathcal{W}_{\bar{\eta}}\theta_{3,y}(b,a)\right|\mathrm{d}b\mathrm{d}a.
$$
As in the proof of Theorem \ref{th1}, the above integrand can be uniformly estimated by
$C(1+|b|^{2})^{-n}.$ This completes the proof (\ref{sp}).
\end{proof}

We obtain the following useful properties.
\begin{corollary} \label{cZ1} The following properties hold:
\begin{itemize}
\item [(i)]$\partial^{m}: C^{\alpha, L}_{\ast} (\mathbb{R}^n)\rightarrow C^{\alpha -\left|m\right|, L}_{\ast}
(\mathbb{R}^n)$ is continuous, for any $m\in\mathbb{N}^{n}$.
\item [(ii)] If $f\ast\phi\in L^{\infty}(\mathbb{R}^{n})$ and $ \partial_{j}f\in C^{\alpha - 1, L}_{\ast}
(\mathbb{R}^n) $ for $j=1,\dots,n,$ then $ f \in
C^{\alpha, L}_{\ast} (\mathbb{R}^n).$
\item [(iii)] The mapping $(1-\Delta)^{\beta/2}$ is an isomorphism of the space $C^{\alpha, L}_{\ast} (\mathbb{R}^n)$ onto $C^{\alpha -\beta, L}_{\ast}
(\mathbb{R}^n)$, for arbitrary $\alpha,\beta\in\mathbb{R}$.
\end{itemize}
\end{corollary}

\begin{proof}
(i) It is enough to consider $\partial_{j}$. We have that $\partial_{j} f\ast \phi=f\ast \partial_{j}\phi$ and $\mathcal{W}_{\psi}\partial_{j}{f}(x,y)=-y^{-1}\mathcal{W}_{\partial_{j}\psi}{f}(x,y)$. Thus, the result follows at once by applying (\ref{fp}) with $\theta=\partial_{j}\phi$ and (\ref{sp}) with $\theta=\partial_{j}\psi$.

(ii) If $(\phi,\psi)$ is a LP-pair, so is $(\phi,\Delta \psi)$. Note that our assumption and (i) imply that $\Delta f\in C^{\alpha-2,L}_{\ast}(\mathbb{R}^{n})$. In view of Proposition \ref{propindZ}, it remains to observe that
$$
\frac{1}{y^\alpha L(y)} \mathcal{W}_{\Delta \psi} f (x, y)
= \frac{1}{y^{\alpha-2} L(y)}\mathcal{W}_{\psi} (\Delta f) (x, y).
$$

(iii) Since $(1-\Delta)^{-\beta/2}$ is the inverse of $(1-\Delta)^{\beta/2}$, it suffices to show that $(1-\Delta)^{\beta/2}$ maps continuously $C^{\alpha, L}_{\ast} (\mathbb{R}^n)$ into $C^{\alpha -\beta, L}_{\ast}
(\mathbb{R}^n)$. Using (\ref{fp}) with $\theta=(1-\Delta)^{\beta/2}\phi$, we obtain that $\left\|(1-\Delta)^{\beta/2}f \ast\phi\right\|_{L^{\infty}}\leq C \left\|f\right\|_{C^{\alpha,L}_{\ast}}$. We also have
\begin{align*}
\mathcal{W}_{\psi}(1-\Delta)^{\beta/2}f(x,y)&=\frac{1}{(2\pi)^{n}y^{\beta}}\left\langle \hat{f}(\xi),e^{ix\cdot \xi}(y^{2}+\left|y\xi\right|^{2})^{\beta/2}\overline{\hat{\psi}}(y\xi)\right\rangle
\\
&
=\frac{1}{y^{\beta}}\mathcal{W}_{\theta_{y}}f(x,y),  \ \ \ \mbox{where }\theta_{y}=(y^{2}-\Delta)^{\beta/2}\psi.
\end{align*}
Finally, we can apply (\ref{sp}) because $\mathfrak{B}=\left\{(y^{2}-\Delta)^{\beta/2}\psi\right\}_{y\in(0,1]}$ is a bounded set in $\mathcal{S}(\mathbb{R}^{n})$ and $\mu_{m}(\theta_{y})=0$ for each multi-index $\left|m\right|\leq[\alpha]$.
\end{proof}

We can also use Proposition \ref{pLP} to show that $C_{\ast}^{\alpha,L}(\mathbb{R}^{n})$ is a Banach space, as stated in the following proposition.

\begin{proposition}
\label{pBZ} The space $C_{\ast}^{\alpha,L}(\mathbb{R}^{n})$ is a Banach space when provided with the norm $(\ref{eq7})$.
\end{proposition}
\begin{proof}
Let $\eta$, $\varphi$, $\theta_{1}$, $\theta_{2}$ be as in the statement of Proposition \ref{pLP}. Suppose that $\left\{f_{j}\right\}_{j=0}^{\infty}$ is a Cauchy sequence in $C_{\ast}^{\alpha,L}(\mathbb{R}^{n})$. Then, there exist continuous functions $g\in L^{\infty}(\mathbb{R}^{n})$ and $G$ defined on $\mathbb{R}^{n}\times (0,1]$ such that $f_{j}\ast \phi\to g$ in $L^{\infty}(\mathbb{R}^{n})$ and
$$
\lim_{j\to\infty}\sup_{y\in(0,1]}\sup_{x\in\mathbb{R}^{n}}\frac{1}{y^{\alpha}L(y)}\left|\mathcal{W}_{\psi}f_{j}(x,y)-G(x,y)\right|=0.
$$
We define the distribution $f\in\mathcal{S}'(\mathbb{R}^{n})$ whose action on test functions $\theta\in\mathcal{S}(\mathbb{R}^{n})$ is given by
$$
\left\langle f,\theta\right\rangle:=\int_{\mathbb{R}^{n}}g(t)\theta_{1}(t)\mathrm{d}t\: +\frac{1}{c_{\psi,\eta}}\int_{0}^{1}\int_{\mathbb{R}^{n}}G(x,y)\mathcal{W}_{\bar{\eta}}\theta_{2}(x,y)\frac{\mathrm{d}x\mathrm{d}y}{y}.
$$
Since the $f_{j}$ have the representation (\ref{eqLP}), we immediately see that $f_{j}\to f$ in $\mathcal{S}'(\mathbb{R}^{n})$. Thus, pointwisely,
$$(f\ast\phi)(t)=\lim_{j\to\infty}(f_{j}\ast \phi)(t)=g(t)$$
and
$$
\mathcal{W}_{\psi}f(x,y)=\lim_{j\to\infty}\mathcal{W}_{\psi}f_{j}(x,y)=G(x,y).
$$
This implies that $\lim_{j\to\infty}\left\|f-f_{j}\right\|_{C^{\alpha,L}_{\ast}}=0$, and so  $C^{\alpha,L}_{\ast}(\mathbb{R}^{n})$ is complete.
\end{proof}

We have arrived to the main and last result of this section. It provides the $L$-H\"{o}lderian characterization of the $L$-Zygmund spaces of positive exponent. We shall use in its proof a technique based on the Tauberian theorem for pointwise weak regularity of vector-valued distributions, explained in Subsection \ref{sphs}. We denote below by $C_{b}(\mathbb{R}^{n})$ the Banach space of continuous and bounded functions.
\begin{theorem} \label{th5} Let $\alpha>0$.
\begin{itemize}
\item [(a)] If $\alpha \notin \mathbb{N}$, then $C^{\alpha, L}_{\ast} (\mathbb{R}^n) = C^{\alpha, L}
(\mathbb{R}^n).$ Moreover, the norms $(\ref{eq7})$ and $(\ref{eqnH})$ are equivalent.
\smallskip

\item [(b)] If $ \alpha = p+1 \in \mathbb{N}, $ then
$C^{p+1, L}_{\ast} (\mathbb{R}^n) $ consists of functions with continuous derivatives up to order $p$ such that
\begin{equation} \label{eq9}
\sum_{\left|m\right|\leq p}\|\partial^{m}f\|_{L^{\infty}} + \sum_{\left|m\right|=p}\underset{0 < |h| \leq 1}{\sup_{t \in \mathbb{R}^n}} \frac{|\partial^{m}f(t + h) + \partial^{m}f(t - h)-2 \partial^{m}f(t)|}{|h|
L(|h|)} < \infty.
\end{equation}
In addition, $(\ref{eq9})$ produces a norm that is equivalent to $(\ref{eq7})$.
\end{itemize}
\end{theorem}

\begin{proof} Observe that the $L$-H\"{o}lderian type norm (resp. (\ref{eq9})) is clearly stronger than (\ref{eq7}). Thus, if we show the equality of the spaces in (a) and (b), the equivalence of norms would be a direct consequence of the open mapping theorem.

Suppose that $f\in C^{\alpha,L}_{\ast}(\mathbb{R}^{n})$. Consider the $C_{b}(\mathbb{R}^{n})$-valued distribution $ \mathbf{f} \in {\mathcal
S}'(\mathbb{R}_{t}^{n}, C_{b}(\mathbb{R}_{\xi}^{n})) $ given by
$\mathbf{f} (t) (\xi) := f(t + \xi)$, i.e., the one whose action on test functions is given by
$$\langle \mathbf{f}(t), \theta(t) \rangle (\xi) =
\langle f(t + \xi), \theta(t) \rangle = (f \ast \check{\theta})(\xi), \ \ \ \theta \in {\mathcal S} (\mathbb{R}^n),\: \xi\in\mathbb{R}^{n}.$$
It does take values in $C_{b}(\mathbb{R}^{n})$ because of (\ref{fp}). Clearly,
$
\mathcal{W}_{\psi}\mathbf{f} (x, y)(\xi) = \mathcal{W}_{\psi}f (x + \xi, y).
$
By (\ref{eq7}) and Potter's estimate \cite[p. 25]{bingham}, we have that
\begin{equation*}
\label{estw}
\left\|\mathcal{W}_{\psi}\mathbf{f}(\varepsilon x,\varepsilon y)\right\|_{C_{b}(\mathbb{R}^{n})}\leq C \varepsilon^{\alpha}L(\varepsilon)y^{\alpha-1} \ \ \ \mbox{ for all } \varepsilon\in(0,1),\ (x,y)\in\mathbb{R}^{n}\times(0,1].
\end{equation*}
Therefore, the Tauberian Theorem \ref{tauberian} yields $
\mathbf{f} \in C_{\ast,w}^{\alpha,L}(0, C_b(\mathbb{R}^{n}))$. Now, the \L ojasiewicz point values $\partial^{m}\mathbf{f}(0)=v_{m}\in C_{b}(\mathbb{R}^{n})$ exist, distributionally, for $\left|m\right|<\alpha$. It explicitly means that for all $\theta\in\mathcal{S}(\mathbb{R}^{n})$
$$
\lim_{\varepsilon \to 0^+} \partial^{m} f \ast \check{\theta}_{\varepsilon} =
\lim_{\varepsilon \to 0^+} \langle \partial^{m}\mathbf{f}(\varepsilon t),
\theta (t) \rangle= \mu_{0}(\theta)v_m \ \ \ \mbox{in }C_{b}(\mathbb{R}_{\xi}^{n}),
$$
where $\check{\theta}_{\varepsilon}(t)=\varepsilon^{-n}\theta\left(-t/\varepsilon\right)$. If we now take $\theta$ with $\mu_{0}(\theta)=1$, we then conclude that $\partial^{m}f=v_{m}\in C_{b}(\mathbb{R}^{n})$ for each $\left|m\right|<\alpha$. It remains in both cases to deal with the estimates for $\partial^{m}f$; notice that $\partial^{m}\mathbf{f}\in C^{\alpha-[\alpha],L}_{\ast,w}(0,C_{b}(\mathbb{R}^{n}))$ when $\left|m\right|=[\alpha]$ and $\partial^{m}\mathbf{f}\in C^{1,L}_{\ast,w}(0,C_{b}(\mathbb{R}^{n}))$ when $|m|=p$. We now divide the proof into two cases:

\emph{Case $\alpha\notin\mathbb{N}$.} Fix a multi-index $\left|m\right|=[\alpha]$ . It suffices to show
$$
\sup_{0<\left|x-t\right|<1} \frac{\left|\partial^{m}f(x)-\partial^{m}f(t)\right|}{\left|x-t\right|^{\alpha-[\alpha]}L(\left|x-t\right|)}<\infty.
$$
We had already seen that
$\partial^{m}\mathbf{f}(t)(\xi)=\partial^{m}f(t+\xi)\in
C^{\alpha-[\alpha],L}_{w}(0,C_{b}(\mathbb{R}_{\xi}^{n}))=C^{\alpha-[\alpha],L}_{\ast,w}(0,C_{b}(\mathbb{R}_{\xi}^{n}))$, i.e.,
$$\mu_{0}(\theta)\partial^{m} f(\xi)-\int_{\mathbb{R}^{n}}\partial^{m}f(\xi+\varepsilon
t)\theta(t)\mathrm{d}t=O(\varepsilon^{\alpha-[\alpha]}L(\varepsilon)), \ \ \ \varepsilon\to0^{+},$$
in the space $C_{b}(\mathbb{R}_{\xi}^{n})$, for each $\theta\in\mathcal{S}(\mathbb{R}^{n})$. Hence, if
$0<\left|h\right|\leq1$, we choose $\theta$ as before ($\mu_{0}(\theta)=1$),
and we use the fact that
$\left\{\theta-\theta(\:\cdot\:-\omega): \left|\omega\right|=1\right\}$
is compact in $\mathcal{S}(\mathbb{R}^{n})$; we then have
\begin{align*}
&\sup_{\xi\in\mathbb{R}^{n}}\left|\partial^{m}f(\xi+h)-\partial^{m}f(\xi)\right|
\leq2\sup_{\xi\in\mathbb{R}^{n}}\left|\partial^{m}f(\xi)-\int_{\mathbb{R}^{n}} \partial^{m}f(\left|h\right|t+\xi)\theta(t)\mathrm{d}t\right|
\\
&+\sup_{\xi\in\mathbb{R}^{n}}\left|\int_{\mathbb{R}^{n}}\partial^{m}f(\xi+\left|h\right|t)(\theta(t)-\theta\left(t-|h\right|^{-1}h))\mathrm{d}t\right|
=O(\left|h\right|^{\alpha-[\alpha]}L(\left|h\right|)),
\end{align*}
and this completes the proof of (a).

\emph{Case $\alpha=p+1\in \mathbb{N}$.} The proof is similar to that of (a). Fix now $\left|m\right|=p$. We now have $\partial^{m}\mathbf{f}\in
C^{1,L}_{\ast,w}(0,C_{b}(\mathbb{R}^{n}))$, which, as commented in Subsection \ref{sphs}, yields the distributional expansion
\begin{equation}
\label{eq8} \partial^{m}\mathbf{f} (\varepsilon t)(\xi) = \partial^{m}f(\xi)+
\varepsilon \sum_{j = 1}^{n} t_{j} \mathbf{c}_{j} (\varepsilon,\xi) +O
\left(\varepsilon L(\varepsilon)\right),  \ \ \ 0<\varepsilon\leq 1,
\end{equation}
in $\mathcal{S}'(\mathbb{R}^{n}_{t},C_{b}(\mathbb{R}^{n}_{\xi}),$ where the $\mathbf{c}_{j}(\varepsilon,\:\cdot\:)$ are continuous $C_b(\mathbb{R}_{\xi}^{n})$-valued functions in $\varepsilon$. We apply (\ref{eq8}) on
a test function $ \theta \in \mathcal{S} (\mathbb{R}^n), $ with $\mu_{0}(\theta) = 1,$ and
$\int_{\mathbb{R}^n} t_j \theta(t)\:\mathrm{d}t = 0$ for $j = 1, \dots, n$, so we get
\begin{equation} \label{eq10}
\partial^{m}f (\xi) = \int_{\mathbb{R}^{n}}\partial^{m}f (\xi + |h| t) \theta(t) \mathrm{d}t + O \left(|h|
L(|h|)\right), \ \ \ 0<\left|h\right|\leq 1,
\end{equation}
uniformly in $\xi\in\mathbb{R}^{n}$. Since
$
\left\{\theta_{\omega}:=\theta(\:\cdot\:+\omega)+\theta(\:\cdot\:-\omega)-2\theta: \left|\omega\right|=1\right\}
$
is compact in $\mathcal{S}(\mathbb{R}^{n})$ and $\mu_{m}(\theta_{\omega})=0$ for $\left|m\right|\leq 1$, the relations (\ref{eq8})  and (\ref{eq10}) give
\begin{align*}
&\sup_{\xi\in\mathbb{R}^{n}}|\partial^{m}f (\xi + h) + \partial^{m}f (\xi - h) - 2\partial^{m}f (\xi)|
\\
&\leq 3\sup_{\xi\in\mathbb{R}^{n}}\left|\partial^{m}f(\xi)-\int_{\mathbb{R}^{n}} \partial^{m}f(\left|h\right|t+\xi)\theta(t)\mathrm{d}t\right|
\\
&
 \ \ \ +\left|\int_{\mathbb{R}^{n}} \partial^{m}f(\xi + |h| t) (\theta(t+\left|h\right|^{-1}h)+\theta(t-\left|h\right|^{-1}h)-2\theta(t)) \mathrm{d}t
\right|
\\
&
=O(\left|h\right|L(\left|h\right|)), \ \ \ 0<\left|h\right|\leq 1,
\end{align*}
as claimed.
\end{proof}

\end{document}